\documentclass[11pt]{amsart}

\usepackage[latin1]{inputenc}

\usepackage[all]{xy}

\usepackage{cases}
\usepackage{hyperref}

\usepackage{color}

\usepackage{enumerate} \usepackage{float} 
\usepackage{amssymb}
\newcommand{\Z}{{\mathbb Z}} \newcommand{\Q}{{\mathbb Q}}

 \newcommand{\F}{{\mathbb F}}

\usepackage{tikz}
\usepackage{dsfont}
\usetikzlibrary{matrix}

\newcommand{\Hom}{\operatorname{Hom}\nolimits}

\newcommand{\GL}{\operatorname{GL}\nolimits}

\newcommand{\A}{\ifmmode{\mathcal{A}}\else${\mathcal{A}}$\fi}
\newcommand{\B}{\ifmmode{\mathcal{B}}\else${\mathcal{B}}$\fi}
\newcommand{\C}{\ifmmode{\mathcal{C}}\else${\mathcal{C}}$\fi}
\newcommand{\D}{\ifmmode{\mathcal{D}}\else${\mathcal{D}}$\fi}
\newcommand{\G}{\ifmmode{\mathcal{G}}\else${\mathcal{G}}$\fi}
\newcommand{\I}{\ifmmode{\mathcal{I}}\else${\mathcal{I}}$\fi}
\newcommand{\J}{\ifmmode{\mathcal{J}}\else${\mathcal{J}}$\fi}
\newcommand{\K}{\ifmmode{\mathcal{K}}\else${\mathcal{K}}$\fi}
\renewcommand{\O}{\ifmmode{\mathcal{O}}\else${\mathcal{O}}$\fi}
\renewcommand{\P}{\ifmmode{\mathcal{P}}\else${\mathcal{P}}$\fi}
\newcommand{\U}{\ifmmode{\mathcal{U}}\else${\mathcal{U}}$\fi}
\newcommand{\M}{\ifmmode{\mathcal{M}}\else${\mathcal{M}}$\fi}
\newcommand{\N}{\ifmmode{\mathcal{N}}\else${\mathcal{N}}$\fi}
\newcommand{\Ss}{\ifmmode{\mathcal{S}}\else${\mathcal{S}}$\fi}
\newcommand{\T}{\ifmmode{\mathcal{T}}\else${\mathcal{T}}$\fi}
\newcommand{\Ff}{\ifmmode{\mathcal{F}}\else${\mathcal{F}}$\fi}
\newcommand{\Ll}{\ifmmode{\mathcal{L}}\else${\mathcal{L}}$\fi}
\newtheorem{Thm}{Theorem}
\newtheorem{Prop}[Thm]{Proposition}

\newtheorem{Lem}[Thm]{Lemma}

\theoremstyle{definition}

\newtheorem{Ex}[Thm]{Example}

\theoremstyle{remark}
\newtheorem{Not}{Notation}

\author{Oihana Garaialde Oca\~{n}a}
\address{Mathematisches Institut ,
Heinrich-Heine Universit\"at D\"usseldorf, 
 40225, D\"usseldorf, Germany.
}
\email{oihana.garayalde@ehu.es}

\theoremstyle{plain}

\title{Cohomology of uniserial $p$-adic space groups with cyclic point group} 

\begin{document}

\maketitle

\begin{abstract}
Let $p$ be a fixed prime number and let $R$ denote a uniserial $p$-adic space group of dimension $d_x=(p-1)p^{x-1}$ and with cyclic point group of order $p^x$. In this short note we prove that all the quotients of $R$ of size bigger than or equal to $p^{d_x+x}$ have isomorphic mod $p$ cohomology groups. In particular, we show that the cohomology groups of sufficiently large quotients of the unique maximal class pro-$p$ group are isomorphic as $\F_p$-modules.
\end{abstract}

\section{Introduction}

Let $G$ be a $p$-group of order $p^n$ and nilpotency class $m$. Then, the coclass of $G$ is $n-m$. In \cite{Carlson}, J. F. Carlson shows that there are finitely many isomorphism types of cohomology algebras over the finite field $\F_2$ for $2$-groups of fixed coclass. His proof uses Leedham-Green's classification theorem \cite{LeedGreen}. In the same paper \cite{Carlson}, Carlson also conjectures that an analogous result should hold for odd primes and he claims that one of the key steps is to prove that all the quotients of the unique pro-$p$ group of maximal class have isomorphic cohomology algebras \cite[Question 6.1]{Carlson}. In \cite{DGG}, the authors  prove that there are finitely many algebras realizing such cohomology groups \cite[Proposition 5.7]{DGG} while in \cite{Graham}, G. Ellis proves that certain periodically constructed quotients of crystalographic groups have isomorphic mod-$p$ cohomology groups.

In this short note, we shall show that the cohomology of all `big enough' quotients of the unique pro-$p$ group of maximal class are isomorphic as graded $\F_p$-modules (see Proposition \ref{prop: pointgroupoforderp}). In fact, we will prove that all `big enough' quotients of uniserial $p$-adic space groups with cyclic point group are isomorphic as graded $\F_p$-modules.

More precisely, let $G$ denote the standard uniserial $p$-adic space group of dimension $d_x=(p-1)p^{x-1}$, that is, $R$ is a pro-$p$ group that fits into a split extension
\[
0\to T\to G\to W(x)\to1,
\]
where $T$ is the translation group 
\[
T:=\overset{p^{x-1}}\oplus K:= \overset{p^{x-1}}\oplus \Z_p^{p-1},
\]
and up to conjugation, $W(x)$ is the iterated wreath product,  
$$
W(x)=C_p\wr\overbrace{C_p\wr \dots\wr C_p}^{x-1}.
$$
called the point group of $G$.

Let $P=\langle \theta \rangle\cong C_{p^{x}}$
be a subgroup of $W(x)$ acting on $T$ by restricting the action of $W(x)$ to the cycle $\theta=(1\dots p^{x})$. The action of $\theta$ on $T$ is uniserial and for each $i \geq 0$, there exists a unique $\theta$-invariant subgroup $N_i \leq T$ of index $p^{i+d_x}$ (see \cite[Definition 4.2.10 and Proposition 4.2.11]{LeedGreen}). This family of subgroups satisfy the following properties: 
\begin{itemize}
\item[(a)] $N_0 = pT$,
\item[(b)] $[\theta, N_i]=N_{i+1}$, 
\item[(c)] $N_i \subset N_{i+1}$.
\end{itemize}
We shall use the following notation throuroght the manuscript.

\begin{Not}\label{Notation}
Let $R$ denote the uniserial $p$-adic space group with translation group $T$ and point group $P=C_{p^x}$. Let $T_i$ denote the quotient $T/N_i$, let $R_i$ be the quotient group $R/N_i$ and analogously, let $K_0$ denote the quotient $K/pK$. Note that $K_0\cong C_p^{p-1}$ and $T_0\cong C_p^{d_x}$. 

Also, let $\text{nil}(H^*(G;\F_p))$ denote the ideal generated by the nilpotent elements of $H^*(G;\F_p)$. Then, the quotient $H^*(G;\F_p)/\text{nil}(H^*(G;\F_p))$ is called the reduced (mod-$p$) cohomology and it is written by $H^*_{\text{red}}(G;\F_p)$.
\end{Not}

We prove the following result. 

\begin{Thm}\label{isoasmodules}
For all $i\geq 0$, there is an isomorphism
\[
H^*(R_i)\cong H^*(R_0)
\]
as graded $\F_p$-modules.
\end{Thm}


\section{Proof of the Theorem}

Let $G$ be the standard uniserial $p$-adic space group of dimension $d_x=p^{x-1}(p-1)$ with translation group $T=\overset{p^{x-1}}\oplus K:= \overset{p^{x-1}}\oplus \Z_p^{p-1}$ and point group 
\[
W(x)=C_p\wr\overbrace{C_p\wr \dots\wr C_p}^{x-1}.
\]
Recall that the left-most copy of $C_p$ acts via the companion matrix of the polynomial $y^{p-1}+y^{p-2}+\dots+y+1$ and the remaining $x-1$ copies of $C_p$ act by permutation matrices as $C_p\wr \overset{x-1}{\dots}\wr C_p$ is the Sylow $p$-subgroup of $\Sigma_{p^{x-1}}$ \cite{LeedGreen}. Let $P=\langle \theta \rangle\cong C_{p^{x}}$ be a subgroup of $W(x)$ acting on $T$ by restricting the action of $W(x)$ to the cycle $\theta=(1\dots p^{x})$. We shall prove Theorem \ref{isoasmodules}.

\begin{proof} Let $Y=\Hom(K/pK,\F_p)=\Hom(K_0,\F_p)$. By abusing the notation, let $\Lambda^* (Y)$  denote both the exterior algebra and the differential graded algebra equipped with the zero differential. There is a morphism of differential graded modules
\[
\varphi_o  :\Lambda^* (Y) \longrightarrow C^* (K_0; \F_p ),
\]
described in \cite[Page 7]{Taelman} or in \cite[Section 5.1]{DGG}, that induces an isomorphism in the ideal of nilpotent elements of the cohomology algebra $H^*(K_0;\F_p)$ (see \cite[Lemma 5.3]{DGG}). Following Notation \ref{Notation} above and for all $i\geq 0$, let 
\[
\varphi_i \colon \overset{p^{x-1}}\otimes\Lambda^*(Y) \to C^*(T_i;\F_p)
\]
be a cochain map defined by the composition of the following morphisms:
\[
\overset{p^{x-1}}\otimes\Lambda^*(Y)\overset{\otimes \varphi_o}\rightarrow \overset{p^{x-1}}\otimes C^* (K_0; \F_p ) \overset{\eta}\to C^*(T_0;\F_p)\overset{\text{inf}}\rightarrow C^*(T_i;\F_p), 
\]
where the inflation map $\text{inf}$ is induced by considering the projections $\pi_i: T_i \to T/pT$  and $\eta$ is the product in the Kunneth formula. Note that $\eta$ induces an isomorphism in cohomology since we are working over a filed \cite[pages 18, 32]{LEvens91} and that $\text{inf}$ induces an isomorphism in the ideal generated by the nilpotent elements. Thus, $\varphi_i$ induces an isomorphism in the ideal generated by the nilpotent elements of $H^*(T_i;\F_p)$. Moreover, as $\text{inf}$ is defined at the level of groups, it is $W(x)$-invariant and by \cite[Proposition 5.10]{DGG} $\otimes \varphi_o$ is also $W(x)$-invariant. It remains to show that $\eta$ is also $W(x)$-invariant. To that aim, we shall show that for $q=(a_1, \dots, a_{p^{x-1}}, \sigma)\in W(x)$, the following equality holds:
\[
q \cdot \eta(f)(z)=\eta(q \cdot f)(z),
\]
for $f \in \overset{p^{x-1}}\otimes C^m(K_0;\F_p)$ 
and $m \geq 0$. On the one hand,
\begin{align*}
\eta(q \cdot f)(z)&= \eta(q \cdot (1 \otimes \cdots \otimes f_i \otimes \cdots \otimes 1))(z_0, \dots, z_m)\\
&=\eta(\sigma\cdot (1 \otimes \cdots \otimes p_i \cdot f_i \otimes \cdots \otimes 1))\\
&=a_i \cdot f_i(z_{0,\sigma(i)}, \dots, z_{m, \sigma(i)}))=f_i(a_i^{-1}z_{0,\sigma(i)}, \dots, a_i^{-1}z_{m,\sigma(i)}).
\end{align*}

On the other hand, 
\begin{align*}
q\cdot \eta(f)(z)&=q \cdot f_i(z_{0,i}, \dots, z_{m,i})=\sigma \cdot a_i\cdot f_i(z_{0,i}, \dots, z_{m,i})\\
&=\sigma \cdot f_i(a_i^{-1}z_{0,i}, \dots, a_i^{-1}z_{m,i})=f_i(a_i^{-1}z_{0,\sigma(i)}, \dots, a_i^{-1}z_{m,\sigma(i)}).
\end{align*}
Thus, $\eta$ is also $W(x)$-invariant and hence, $\varphi_i$ is $W(x)$-invariant. 

Consider the $\Q_p$-vector space $\tilde{T}=T\otimes_{\Z_p}\Q_p$ and notice that the action of $W(x)$ on $T$ extends to $\tilde{T}$.  For all $i\geq 0$, the group $\tilde{T}_i:=\tilde{T}/N_i$ is isomorphic to the $d_x$-fold product of the Pr\"uffer group $C^{d_x}_{p^{\infty}}=C_{p^{\infty}}\times\overset{d_x}\cdots\times C_{p^{\infty}}$, where $C_{p^{\infty}}$ is the direct limit of all the cyclic groups $C_{p^n}$. We define, for each $i\geq 0$, a morphism of differential graded algebras
$$\psi_i: C^* (\tilde{T}_i,\F_p)\longrightarrow  C^* (T_i,\F_p),$$
induced by the natural  inclusion $T_i\to \tilde{T}_i$. This cochain map commutes with $W(x)$ and gives an isomorphism in the reduced part of the cohomology ring (see \cite[page 17]{DGG}). Then, the morphism
\begin{equation}
\label{morab}
\varphi_i\otimes \psi_i: \overset{p^{x-1}}\otimes\Lambda^* (Y)\otimes C^* (\tilde{T}_i;\F_p)\longrightarrow C^* (T_i; \F_p ) \otimes C^* (T_i; \F_p ) \stackrel{\cup}\to C^* (T_i; \F_p ),
\end{equation}
is a $W(x)$-invariant quasi-isomorphism of differential graded $\F_p$-modules. In particular, $\varphi_i \otimes \psi_i$ is $\theta$-invariant.

Finally, we shall show that there is a quasi-isomorphism between 
$$
\overset{p^{x-1}}\otimes \Lambda(Y)\otimes C^*(\tilde{T}_i) \; \; \text{and} \; \; \overset{p^{x-1}}\otimes \Lambda(Y)\otimes C^*(\tilde{T}_{i+1})
$$
for all $i$, to prove the result. Define a map $\delta: \tilde{T} \to \tilde{T}$ that sends an element $a \in \tilde{T}$ to $[a, \theta]=a \theta a^{-1} \theta^{-1}$. Note that $\delta$ is an injective linear map between $p^{d_x}$-dimensional $\Q_p$-vector spaces since $\theta$ acts uniserially and in particular, faithfully on $\tilde{T}$. Hence, $\delta$ is an isomorphism. Also, the equality
$$[\theta , a]^\theta=[\theta ,a^\theta],$$
shows that $\delta$ is $\theta$-invariant. Let $\delta_i$ denote the restriction of $\delta$ to $T_i$. Since $\delta_i(N_i)=[N_i,\theta]=N_{i+1}$, there is a commutative diagram
\begin{equation}
\xymatrix{
\tilde{T} \ar[d] \ar[r]^{\delta} &\tilde{T} \ar[d] \\
\tilde{T}_i \ar[r]^{\tilde{\delta}_i} &\tilde{T}_{i+1},
}
\end{equation}
where $\tilde{\delta}_i$ is an isomorphism. Furthermore,
\begin{equation}
\label{mordelta}
\text{id}\otimes \tilde{\delta}^*_i: \overset{p^{x-1}}\otimes\Lambda^* (Y)\otimes C^* (\tilde{T}_{i+1},\F_p) \longrightarrow \overset{p^{x-1}}\otimes\Lambda^*(Y)\otimes C^* (\tilde{T}_i,\F_p),
\end{equation}
 is an isomorphism of differential graded modules that commutes with the action of $\theta$. Applying \cite[Lemma 2.1]{DGG} in the following diagram
 \begin{equation*}
\xymatrix{
 \overset{p^{x-1}}\otimes \Lambda(Y) \otimes C^*(\tilde{T}_i;\F_p) \ar[rr]^{\varphi_i\otimes \psi_i} & &C^*(T_i;\F_p)\\
 \overset{p^{x-1}}\otimes \Lambda(Y) \otimes C^*(\tilde{T}_{i+1};\F_p) \ar[u]^{\text{id}\otimes\delta_i^*} \ar[rr]^{\; \; \; \; \; \varphi_{i+1}\otimes \psi_{i+1}} & &C^*(T_{i+1};\F_p)
 }
 \end{equation*}
 where all the maps are $\theta$-invariant, we obtain that for all $i \geq 0$,
\[
H^*(R_i)\cong H^*(R_{i+1})
\]
as graded $\F_p$-modules.
\end{proof}

In particular, if $d_x=p-1$, then $R$ is the unique pro-$p$ group of maximal nilpotency class. That is, $R=\Z_p^{p-1}\rtimes C_p$ where $C_p$ acts on $\Z_p^{p-1}$ via the matrix
\begin{small}\begin{equation}
\label{eq:companiontheta}
M=\begin{pmatrix}
1 &0 & \cdots &0 &  -{p \choose 1}\\
1 &1 & \cdots &0 & -{p \choose 2} \\
0 &1 & \cdots &0 & -{p \choose 3} \\
\vdots & \vdots & &\vdots &\vdots \\
0 &0 &\cdots &1 & 1-{p \choose p-1}
\end{pmatrix} \in \GL_{p-1}(\Z).
\end{equation}\end{small}An explicit construction of $R$ can be obtained as follows: Let $\theta$ be a primitive $p^{\text{th}}$ root of unity and let $\Q_p[\theta]$ be the $p^{\text{th}}$ local cyclotomic field with ring of integers $\Z_p[\theta]$. Consider the semi-direct product $C_p\ltimes\Z_p[\theta]$, where the generator $x$ of $C_p$ acts on $\Z_p[\theta]$ by multiplication by $\theta$. Here, $\Z_p[\theta]$ is a $\Z_p$-module of rank $p-1$. This action extends to the $\Q_p$-vector space $\Q_p[\theta]$. We claim that $C_p\ltimes\Z_p[\theta]$ is the pro-$p$ group of maximal nilpotency class. That is, $C_p\ltimes\Z_p[\theta]$ is a uniserial $p$-adic space group of coclass one. We shall show that there exists a $C_p$-invariant filtration $\{T_i\}_{i\in \Z}$ such that $|T_i:T_{i+1}|=p$ and $pT_i=T_{i+p-1}$ for all $i\geq 1$. 

Let $a\in \Q_p[\theta]$ and $x\in C_p$. Then,
\begin{equation}
\label{eq1}
x\cdot a=\theta a=(1+\theta -1)a=a+(\theta -1)a.
\end{equation} 
Since the minimal polynomial for $\theta$ is $(X^p-1)/(X-1)$, it follows that the minimal polynomial for $\theta -1$ is 
\begin{equation}\label{eq: minimalpolyforthetaminusone}
\frac{((X+1)^p-1)}{X}=X^{p-1}+{p \choose 1}X^{p-2}+\ldots +{p \choose p-2}X+{p \choose p-1},
\end{equation}
and the companion matrix for $1-\theta$ is the following one
\begin{equation}
\label{eq:companiontheta-1}
\begin{pmatrix}
0 &0 & \cdots &0 &  -{p \choose 1}\\
1 &0 & \cdots &0 & -{p \choose 2} \\
0 &1 & \cdots &0 & -{p \choose 3} \\
\vdots & \vdots & &\vdots &\vdots \\
0 &0 &\cdots &1 & -{p \choose p-1}
\end{pmatrix} \in \GL_{p-1}(\Z).
\end{equation}
This combined with \eqref{eq1} yields that the action of $\theta$ is given by the matrix
\begin{equation}
\label{eq:companiontheta}
M=\begin{pmatrix}
1 &0 & \cdots &0 &  -{p \choose 1}\\
1 &1 & \cdots &0 & -{p \choose 2} \\
0 &1 & \cdots &0 & -{p \choose 3} \\
\vdots & \vdots & &\vdots &\vdots \\
0 &0 &\cdots &1 & 1-{p \choose p-1}
\end{pmatrix} \in \GL_{p-1}(\Z),
\end{equation}
for an appropriate basis of $\Q_p[\theta]$. 

Consider the ideal $\mathbf{p}=(\theta)$ and construct a chain of ideals
\[
\mathbf{p}\supset \mathbf{p}^2\supset \dots \supset \mathbf{p}^i\supset \dots,
\]
where $|\mathbf{p}^i: \mathbf{p}^{i+1}|=p$ for all $i\geq 1$. It is readily checked that each $\mathbf{p}^i$ is $(\theta-1)$-invariant and in particular, $\theta$-invariant. Moreover, using the relation
\[
(\theta-1)^{p-1}=-{p \choose 1}(\theta-1)^{p-2}-\ldots -{p \choose p-2}(\theta-1)-{p \choose p-1},
\]
obtained from Equation \eqref{eq: minimalpolyforthetaminusone}, we have that $p\mathbf{p}^i=\mathbf{p}^{i+p-1}$ for all $i\geq 1$. Then, $T_i=\mathbf{p}^i$ gives the uniserial filtration.



Hence, $\langle \theta\rangle$ acts uniserially on $\Q_p[\theta]$ and on $\Z_p[\theta]$ and thus, $\Z_p[\theta]\rtimes C_p$ is the unique pro-$p$ group of maximal nilpotency class. In fact, for all $i\in \Z$ the semidirect product $\langle \theta\rangle \ltimes T_i$ is also the unique maximal class pro-$p$ group and for each $n\geq 1$, $T_{-n}/T_0\rtimes \langle \theta\rangle $ is the unique quotient of $T_0\rtimes C_p$ of size $p^{n+1}$. Now, \cite[Proposition 5.8]{DGG} gives the following result.

\begin{Prop}\label{prop: pointgroupoforderp}For all $n\geq p-1$, there is an isomorphism of graded $\F_p$-modules
\[
H^*(T_{-n}/T_0\rtimes \langle \theta\rangle,\F_p)\cong H^*(T_{-p+1}/T_0\rtimes \langle \theta\rangle ,\F_p).
\] 
\end{Prop}

More generally, let $\zeta$ denote a $(p^x)^{\text{th}}$ root of unity and let $\zeta$ act on $\Q_p[\zeta]$ by multiplication by $\zeta$. Then, it can be shown that $\Q_p[\zeta]\rtimes \langle\zeta\rangle$ is a uniserial $p$-adic space group of dimension $(p-1)p^x$ \cite{LeedGreen}. The next result shows that $\Z_p[\zeta]\to \Z_p[\zeta]\rtimes C_{p^x}\to C_{p^x}$ and the uniserial $p$-adic space group with cyclic point group $T\to R\to P=C_{p^x}$ are isomorphic. 

\begin{Lem} There exists a unique uniserial $p$-adic space group with cyclic point group.

\end{Lem}

\begin{proof}
We start by showing that the extension $T\to R\overset{\pi}\to C_{p^x}$ splits. Let $s=\pi^{-1}$ denote a section for the extension. It suffices to show that $\pi\circ s=\text{id}_{C_{p^x}}$ holds. Let $\langle\theta\rangle:=C_{p^{x}}$ and put $\eta:=s(\theta)$ with $\langle\eta\rangle\leq G$. Note that $\pi(\eta^{p^x})=\theta^{p^x}=1$ and thus, $\langle\eta^{p^x}\rangle\leq T$. 

It is readily checked that $\eta^{p^x}$ commutes with $\theta$ and thus, $\langle\eta^{p^x}\rangle\unlhd G$. Now, by \cite[Theorem 7.4.2]{LeedGreen} and \cite[Lemma 8.1]{Shalev94}, $R$ is just infinite. That is, the non-trivial normal subgroups have finite index in $R$. Hence, $\langle \eta^{p^x}\rangle=\{1\}$ and thus, $s$ is a homomorphism such that $\pi \circ s=\text{id}_{C_{p^x}}$ holds. 

It remains to show that the split extension $R=T\rtimes C_{p^x}$ is unique. Note that the faithful irreducible complex representations of $C_{p^x}$ are given by $x\to \theta^{i}$ with $0\leq i<p^x$, where $\theta$ denotes a primitive $(p^{x})^{\text{th}}$ root of unity. Since $\Q_p$ contains no $p^{\text{th}}$ roots of unity, there is a unique faithful irreducible representation of $C_{p^x}$ over $\Q_p$ \cite[Theorem 10.1.13]{LeedGreen} and its minimum polynomial of $\theta$ over $\Q_p$ is the cyclotomic polynomial 
\[
\Phi(x)=x^{(p-1)p^{x-1}}+\dots +1,
\]
which has degree $(p-1)p^{x-1}$. Equivalently, $\tilde{T}$ is the unique indecomposable $\Q_p[C_{p^x}]$-module of rank $(p-1)p^x$. In particular, $R=T\rtimes C_{p^x}$ is the unique uniserial $p$-adic space group of dimension $(p-1)p^x$ with cyclic point group. 
\end{proof}

\begin{Ex} Let $G=(\Z_3\times\Z_3)\rtimes C_3$ denote the pro-$3$ group of maximal nilpotency class, where $C_3$ acts via the integral matrix,
\begin{equation}\label{integralmatrixactionforp3}
M=
\begin{pmatrix}
1 &-3\\
1 &-2
\end{pmatrix}\in \GL_2(\Z).
\end{equation}
Let $\{B(3,r)\}_{r\geq 3}$ denote the family of $3$-groups of maximal class obtained by taking quotients of $G$ described in \cite[Appendix A]{DRV}. Then, $B(3,r)$ fit into the extensions
\[
(C_3^k \times C_3^k)\rtimes C_3 \text{ or } (C_3^k \times C_3^{k-1}) \rtimes C_3,
\]
if $r=2k+1$ or $r=2k$, respectively. By Proposition \ref{prop: pointgroupoforderp}, there is an isomorphism of $\F_3$-modules
\[
H^*(B(3,3);\F_3)\cong H^*(B(3,r);\F_3),
\]
for all $r\geq 3$.
\end{Ex}

\end{document}